\documentclass[11pt]{amsart}

\usepackage{macros,amsaddr,physics}

\renewcommand{\op}{\operatorname}
\numberwithin{equation}{section}

\begin{document}

\title{The local cohomology of vector fields}
\author{Brian R. Williams}
\thanks{Boston University, Department of Mathematics and Statistics}
\email{bwill22@bu.edu}

\maketitle

Let $\op{Vect}(M)$ be the Lie algebra of smooth vector fields on a manifold $M$.
The Lie algebra cohomology of vector fields is the cohomology of the Chevalley--Eilenberg cochain complex 
\[
C^\bu\Vect(M) = \oplus_{k \geq 0} C^k \Vect(M)
\]
where in degree $k$ is the space of continuous $k$-linear totally alternating functional on $\Vect(M)$ and the differential is the Chevalley--Eilenberg differential encoding the Jacobi bracket of vector fields.
The Lie algebra cohomology of vector fields has been studied extensively in the context of characteristic classes of foliations \cite{FuksChar, BottFoliation,  BottSegal, Guillemin, LosikTrivial1, GK, BR}.

An important step in the computation of this cohomology is a computation of the cohomology of the diagonal subcomplex 
\[
C^\bu_{\triangle} \Vect(M) \subset C^\bu\Vect(M)   
\]
which consists of cochains \(\varphi \in C^k\Vect(M)\) satisfying $\varphi(X_1,\ldots,X_k) = 0$ if and only if
\[
\bigcap_{i=1}^k {\rm Supp}(X_i) = \emptyset .
\]
This paper concerns the cohomology of a smaller subcomplex: the complex of \textit{local cochains} of $\op{Vect}(M)$, and the resulting local cohomology.
A cochain \(\varphi \in C^k\Vect(M)\) is said to be local if it can be written as
\[
\varphi (X_1 , \ldots, X_n) = \int_M L(X_1, \ldots, X_n) .
\]
where \(L\) is a density-valued (Lagrangian) function of the vector fields $X_1,\ldots,X_n$ that only depends on the jets of these vector fields.
The local cochain complex of \(\Vect(M)\) will be denoted \(\cloc^\bu\Vect(M)\).
Every local cochain is diagonal, which leads to a sequence of inclusions of cochain complexes
\[
\cloc^\bu \Vect(M) \hookrightarrow \op{C}^\bu_{\triangle} \Vect(M) \hookrightarrow C^\bu \Vect(M) .
\]
A result of Gelfand--Fuks \cite{GFmanifold} implies that the first arrow induces an isomorphism in cohomology, see also \cite{Guillemin,LosikTrivial1}.

Our main result is a characterization of the local cohomology of smooth vector fields on an arbitrary smooth manifold.
Our approach utilizes a relationship between smooth and formal geometry that can be attributed to Gelfand, Kazhdan, Fuks, and many others, for example see 
\cite{GKF,BKformal}.
By similar methods, we compute the local cohomology of the cohomology of holomorphic vector fields on an arbitrary complex manifold.
A note is in order in the complex setting; the Lie algebra of holomorphic vector fields.
In a naive sense, the Lie algebra of holomorphic vector fields is \textit{not} a local Lie algebra (the underlying sheaf is not the $C^\infty$-sections of a vector bundle).
Thus, in order to define the notion of locality (and also the notion of diagonal cohomology) we use a resolution of holomorphic vector fields by vector bundles.

The result about the local cohomology of smooth vector fields follows from the known characterization of the diagonal cohomology of vector fields on a smooth manifold \cite{Guillemin,LosikTrivial1,LosikTrivial2}, see also \cite[\S 2.4]{Fuks}.
Indeed, a result of \cite{GFmanifold} implies that for smooth vector fields the inclusion $C^\bu_{loc} \Vect (M) \hookrightarrow C^\bu_{\triangle} \Vect(M)$ is a quasi-isomorphism. 
The formulation of the local cohomology of holomorphic vector fields is presented in this paper.
For an algebro-geometric definition of the diagonal cohomology of vector fields we refer to \cite{HKgf} where it is used to characterize the entire cohomology of vector fields on a smooth algebraic variety; we expect that their definition to agree with our notion of local cohomology.

We remark on the relationship between local Lie algebra cohomology and Lie algebra cohomology.
For any sufficiently nice sheaf (say, one presented as the sections of a smooth vector bundle) of Lie algebra $\cL$ on a manifold $M$ there is a notion of local cohomology.
The cochain complex $C^\bu_{loc}(\cL)$, introduced in section \ref{sec:dfn}, that computes this is naturally a sheaf on $M$.
At the level of global sections, if $M$ is compact, this complex admits a map $C^\bu_{loc}(\cL(M)) \to C^\bu(\cL(M))$ where $C^\bu(\cL(M))$ is the Lie algebra cohomology of the Lie algebra of global sections.
For $\cL = \Vect(M)$, the sheaf of smooth vector fields, the local cohomology is a shift of the de Rham cohomology of the total space $Y(M)$ of a fibration over $M$ whose fiber over $x \in M$ is isomorphic to the restriction of the universal $U(n)$-bundle over $BU(n)$ to its $2n$-skeleton.
See theorem \ref{thm:smooth} for a precise statement.
The cohomology of $\Vect(M)$, on the other hand, is the cohomology of the space of sections of $Y(M)$, see \cite{Haefliger1, Haefliger2, BottSegal}.
The induced map
\beqn
H^\bu_{loc}(\Vect(M)) \to H^\bu_{Lie}(\Vect(M))
\eeqn
is the composition
\beqn
H^{\bu}_{dR}(Y(M)) \xto{ev^*} H^\bu_{dR}(\op{Sect}(Y(M)) \times M) \xto{\int_M} H^\bu_{dR}(\op{Sect}(Y(M)) [-n]
\eeqn
where the first arrow is pullback along the natural evaluation map 
and the second map is integration along $M$.

Let $X$ be a complex manifold.
In this paper we give a characterization of the local Lie algebra of holomorphic vector fields $\cT(X) \simeq \op{Vect}^{hol}(X)$.
There is again a natural map 
\beqn\label{eqn:hollocalinclusion}
H^\bu_{loc}(\cT(X)) \to H_{Lie}^\bu(\cT(X)) .
\eeqn
As mentioned earlier, we expect that $H^\bu_{loc}(\cT(X))$ agrees with the cohomology of the diagonal complex of vector fields considered in \cite{HKgf}.
One defines the space $Y(X)$ just as in the smooth case; it is a holomorphic fiber bundle over $X$ whose fiber over $x \in X$ is as in the smooth case.
In \cite{HKgf}, it is shown that the cohomology of the space of holomorphic sections $\op{Sect}(Y(X))$ is only isomorphic to $H^\bu(\cT(X))$ in the case that $X$ is affine; in general they produce a map 
\beqn
H^\bu_{dR} (\op{Sect}(Y(X))) \to H^\bu_{Lie} (\cT(X)) 
\eeqn
which we expect our map \eqref{eqn:hollocalinclusion} to factor through.

Interest in the local cohomology of vector fields is motivated, in part, by classical and quantum field theory.
A classical field theory is prescribed by a Lagrangian density depending on the fields, whose associated Euler--Lagrange equations dictate the dynamics of the classical system.
Additionally, local cohomology of local Lie algebras is a home for (local) \textit{anomalies}, which are an important feature of a quantum field theory.
In short, an anomaly describes the failure of a symmetry in a classical field theory to persist to a symmetry at the quantum level.
Similarly to action functionals, certain anomalies can be realized as local cohomology classes.
A diffeomorphism invariant field theory on a manifold \(M\) receives an infinitesimal action by the Lie algebra of smooth vector fields \(\Vect(M)\).
Similarly, a holomorphic field theory on a complex manifold $X$ \cite{BWhol} receives an infinitesimal action by the Lie algebra of holomorphic vector fields.
Anomalies for this infinitesimal action to exist at the quantum level are given by local cohomology classes in \(\cloc^\bu\Vect(M)\), or $\cloc^\bu\Vect^{hol}(X)$; thus motivating the need to understand the resulting cohomology.
The corresponding problem for Riemannian (or Lorentzian) quantum field theories is a classic topic, see \cite{Duff,Deser:1993yx}, for example.
For applications of results of this paper to anomalies in supersymmetric quantum field theory we refer to \cite{BWac}.

\subsection*{Acknowledgements}

I thank Boston University for their support during the preparation of this paper.

\tableofcontents

\newpage

\section{Definitions and main results}
\label{sec:dfn}

In this paper, $M$ denotes a smooth manifold and $X$ denotes a complex manifold.
We work in the $C^\infty$-category, so unless otherwise specified a ``section`` means a smooth section,
a ``differential form'' means a smooth differential form, and so on.
All functions and cochains are complex-valued.

\subsection{Local functionals}

Let \(E \to M\) denote a $\Z$-graded vector bundle on \(M\) and denote by \(\cE\) its sheaf of sections.
We consider the pro vector bundle of $\infty$-jets which we will denote by $\op{jet}(E)$, see \cite{Anderson} or \cite[\S 5.6]{CostelloBook} for instance.
The sheaf of smooth sections of this pro vector bundle carries the natural structure of a $D_M$-module.

\begin{dfn}
Let $E$ be a graded vector bundle on $M$.
The sheaf of \defterm{Lagrangians} on $E$ is the $C^\infty_M$-module
\beqn
\op{Lag} (E) \define \prod_{n > 0} {\rm Hom}_{C^\infty_M} \left(\op{jet} (E) , C^\infty_M\right) .
\eeqn
\end{dfn}

\begin{rmk}
The notation ${\rm Hom}_{C^\infty_M} \left(\op{jet} (E) , C^\infty_M\right)$ refers to the sheaf of continuous linear maps of $C^\infty_M$-modules.
This can be viewed as an ind vector bundle formally dual to the pro vector bundle $\op{jet} (E)$.
The flat connection defining the $D_M$-module structure on $\op{jet} (E)$ endows this sheaf with the structure of a $D_M$-module.
Notice that the constant functionals on $\op{jet} (E)$ do not appear in the definition of $\op{Lag}(E)$, this is mostly for conventional reasons and will not play a huge role in what follows.
\end{rmk}

For any graded vector bundle \(E\), the $C^\infty_M$-module \(\op{Lag}(E)\) has the natural structure of a $D_M$-algebra, induced from the $D_M$-module structure on \(\op{jet}(E)\).

Let \({\rm Dens}_X\) be the right $D_{X}$-module of densities on \(X\).
Given any left $D_M$-module \(V\) one can consider the following $C^\infty_M$-module
\beqn
\op{Dens}_{X} \otimes_{D_{X}} V .
\eeqn
If \(X\) is an oriented smooth manifold and \(V\) is flat, then this agrees with (a shift of) the de Rham complex of \(V\), see below.
For the case at hand, \(V\) is the left $D_M$-module of Lagrangians \({\rm Lag} (E)\) and we have the following definition.

\begin{dfn}
Let $E$ be a vector bundle on $X$.
The $C^\infty_M$-module of \defterm{local functionals} on $X$ is
\beqn
\oloc(E) \define {\rm Dens}_X \otimes_{D_M} {\rm Lag} (E) .
\eeqn
\end{dfn}

Concretely, a section of \({\rm Lag}(E)\) is a sum of functionals of the form
\beqn
\phi \in \cE \mapsto D_{1} \phi_{1} \cdots  D_{n} \phi_{n}
\eeqn
where \(D_i\) are differential operators acting on the bundle \(E\).
Likewise, a section of \(\oloc(E)\) is given as a sum of functionals which send a section \(\phi\) to a class
\beqn
\bigg[D_1 \phi_{1} \cdots  D_{n} \phi_{n} \omega \bigg]
\eeqn
where \(\omega\) is a density on \(X\).
The brackets denotes an equivalence class where two sections are equivalent if they differ up to a total derivative.
For this reason, we will often write such an element using the integration symbol
\beqn
\int D_{1} \phi_{1} \cdots  D_{n} \phi_{n} \omega
\eeqn
where we provide the warning that no actual integration is occurring. \footnote{Of course, unless the section $\phi$ is compactly supported integration over an open subset is ill-defined.}

If $X$ is an oriented smooth manifold, the sheaf of local functionals of $E$ can be expressed using the de Rham complex of the $D_M$-module of Lagrangians.
In this case, ${\rm Dens}_X$ can be replaced by the bundle of top forms $\Omega^{d}_X$ where $d = \dim_{\R}(M)$.
This right $D_M$-module $\Omega^d_M$ has a free resolution of the form
\beqn
\Omega^0 \otimes_{C^\infty_M} D_M [d] \to \cdots \to \Omega^{d-1}_X \otimes_{C^\infty_M} D_M [1] \to \Omega^d_M \otimes_{C^\infty_M} D_M .
\eeqn
Since ${\rm Lag}(E)$ is flat as a $D_M$-module one can use this resolution to obtain a quasi-isomorphism
\beqn\label{derham1}
\oloc(E) \; \simeq \; \Omega^\bu \bigg( X \; , \; {\rm Lag}(E) \bigg)[d] .
\eeqn
We will use this description extensively throughout this paper.
For more details see \cite[lemma 3.5.4.1]{CG2}.
In the unoriented case one would need to use a twisted version of the de Rham complex.

\subsection{Local Lie algebras and cohomology}

The next definition we will need is that of a local dg Lie algebra and its local cohomology.
We refer to \cite{CG2} for more details on the definitions below.

\begin{dfn}
A \defterm{local dg Lie algebra} on a smooth manifold $X$ is a triple $(L, \d, [\cdot , \cdot])$ where:
\begin{itemize}
\item[(i)] $L$ is a $\Z$-graded vector bundle on $X$ of finite total rank;
\item[(ii)] $\d$ is a degree $+1$ differential operator $\d \colon \cL \to \cL$ on the sheaf $\cL=\Gamma(L)$ of smooth sections of~$L$, and
\item[(iii)] $[\cdot, \cdot]$ is a bilinear polydifferential operator
\beqn
[\cdot , \cdot] \colon \cL \times \cL \to \cL
\eeqn
\end{itemize}
such that the triple $(\cL, \d, [\cdot,\cdot])$ carries the structure of a sheaf of dg Lie algebras.
\end{dfn}

Just as in the case of an ordinary graded vector bundle, we can discuss the Lagrangians on a local Lie algebra \(L\).
In this case, \({\rm Lag}(L[1])\) is equipped with the Chevalley--Eilenberg differential \(\d_{\rm CE}\) induced from the Lie algebra structure on \(L\).
In fact, the $\infty$-jet bundle \(\op{jet}(L)\) is a dg Lie algebra object in $D_M$-modules and we have the dg $D_M$-module of reduced Chevalley-Eilenberg cochains
\beqn
\op{C}_{red}^\bu (\op{jet}(L)) = (\op{Lag}(L[1]), \d_{CE}) .
\eeqn
(Notice we look at reduced cochains since we have thrown out the constant functions in the definition of \({\rm Lag}(L[1])\).)
Since \(\d_{CE}\) is compatible with the $D_M$-module structure, this induces a differential on the space of local functionals \(\oloc(L[1])\).

We arrive at the central object of study of this paper.

\begin{dfn}
The \defterm{local Chevalley--Eilenberg cochain complex} of a local Lie algebra~$\cL$ is the sheaf of cochain complexes
\begin{align}
\cloc^\bu(\cL) & \define \left(\oloc(L[1]) , \d_{CE} \right) \\ & = {\rm Dens}_X \otimes_{D_M} \op{C}_{red}^\bu(\op{jet}(L)).
\end{align}
We denote the cohomology of the global sections of this complex of sheaves by $H_{loc}^\bu(\cL (X))$ and refer to it as the local cohomology of $\cL$ over $X$.
\end{dfn}

\subsection{Results on the local cohomology of vector fields}
We now turn to the local Lie algebras of vector fields and the first main results of the paper.
A basic example of a local Lie algebra is the sheaf of smooth vector fields $\op{Vect}_M = \Gamma(M,\T_M)$ on a smooth manifold~$M$.
Our first result is a characterization of the local cohomology of $\Vect_M$ over an $d$-manifold $M$.
To state the result we introduce a formal object.

Denote the Lie algebra of vector fields on the formal $d$-disk, 
as studied by Gelfand and Fuks \cite{GF1,Fuks}, by $\lie{vect}(d)$.
The notation $H_{(red)}^{\bu} (\lie{vect}(n))$ refers to the (reduced) continuous Lie algebra cohomology of the Lie algebra of formal vector fields.
Let $X(d)$ denote the restriction of the universal $U(d)$ bundle to the $2d$-skeleton of $BU(d) = Gr(d,\infty)$.
The following result was first proved for the diagonal cohomology.

\begin{thm}[see also \cite{Guillemin,LosikTrivial1,LosikTrivial2}] \label{thm:smooth}
Let $M$ be a smooth oriented manifold of dimension $d$.
Then
\beqn
H^k_{loc}(\Vect(M)) \; \cong \; H_{red}^{d+k}(Fr^\C_M \times^{U(d)} X_d)
\eeqn
for all $k$.
In particular, if $M$ is parallelizable then
\beqn
H^k_{loc}(\Vect(M)) \; \cong \; \bigoplus_{i=0}^{d} H^i_{dR}(M) \otimes H^{d+ k-i}_{red}(\lie{vect}(d)) .
\eeqn
\end{thm}


Next we turn to the case of complex manifolds and the Lie algebra of holomorphic vector fields.


Let $X$ be a complex manifold and denote by $\T^{1,0}_X$ the holomorphic tangent bundle.
Consider its Dolbeault complex
\beqn
\cT_X \define \Omega^{0,\bu}(X , \T^{1,0}_X)
\eeqn
This is a sheaf of cochain complexes (in fact, it is an elliptic complex) where the differential is the $\dbar$-operator.
Moreover, this sheaf of cochain complexes is equipped with a bracket $[\cdot, \cdot]$ which extends the Lie bracket of vector fields.
This endows $\cT_X$ with the structure of a local Lie algebra.
The assignment $X \mapsto \cT_X$ defines a sheaf on the site of complex manifolds which we will denote simply by $\cT$.

The Dolbeault complex of a holomorphic vector bundle is a resolution for its sheaf of holomorphic sections.
Note that the sheaf of holomorphic vector fields is {\em not} a local Lie algebra since it is \textit{not} the space of $C^\infty$-sections of a vector bundle.
Therefore, to capture the notion of holomorphic vector fields using local Lie algebras it is necessary to consider this resolution $\cT$.
Indeed, if $\cT^{hol} = \Gamma^{hol}(\T^{1,0})$ denotes the sheaf of holomorphic vector fields, the embedding $\cT^{hol} \hookrightarrow \cT$ is a quasi-isomorphism.

Our next result pertains to the local cohomology of holomorphic vector fields and is the content of section \ref{sec:holomorphic}.
It closely parallels the result in the smooth case above.

\begin{thm} \label{thm:hol}
Let $X$ be a complex manifold of complex dimension $n$ and let $Fr_X$ be the principal $U(n)$-bundle of frames of the holomorphic tangent bundle.
Then
\beqn
H^k_{loc}(\cT(X)) \; \cong \; H_{dR}^{k+n}(Fr_X \times^{U(n)} X(n)). 
\eeqn
In particular, if $X$ has trivializable tangent bundle then
\beqn
H^k_{loc}(\cT(X)) \; \cong \; \bigoplus_{i=0}^{2n} H^i_{dR}(X) \otimes H^{2n + k-i}_{red}(\lie{vect}(n)) .
\eeqn
\end{thm}


In \S \ref{sec:descent} we specialize to the flat case \(X = \C^n\). 
We extract explicit representatives for cohomology classes.
In this case, the local cohomology reduces to a shift of the Gelfand--Fuks cohomology
\beqn
H^\bu_{loc}(\cT(\C^n)) \cong H_{red}^\bu(\lie{vect}(n))[2n]
\eeqn
by theorem \ref{thm:hol}.
We will describe an explicit quasi--isomorphism
\beqn\label{delta}
\delta \colon C_{red}^\bu(\lie{vect}(n))[2n] \xto{\simeq} \cloc^\bu(\cT(\C^n)) .
\eeqn
The map \(\delta\) is constructed using the method of topological descent.
It utilizes the existence of two types of degree \((-1)\) endomorphisms on the complex of local functionals that we denote \(\eta_i\) and \(\Bar{\eta}_i\).

These operators can be described heuristically as follows.
On \(\C^n\), the local cochain complex \(\cloc^\bu (\cT(\C^n))\) receives an action by the Lie algebra of translations spanned by the constant holomorphic vector fields \(\partial_{z_i}\) and constant anti-holomorphic vector fields~\(\partial_{\zbar_i}\).
The action of this Lie algebra is homotopically trivial.
The operator \(\eta_i\) provides a homotopy for the holomorphic vector field \(\partial_{z_i}\) and \(\Bar{\eta}_i\) provides a homotopy for~\(\partial_{\zbar_i}\).

Using these homotopies, we can give a description of the map \(\delta\) in (\ref{delta}).
Notice that there is a map of Lie algebras \(j : \cT^{hol}(D^n) \to \lie{vect}(n)\) which records the Taylor expansion of a vector field at \(0 \in D^n\).
Here, \(\cT^{hol}(D^n)\) denotes the Lie algebra of holomorphic vector fields on an $n$-disk centered at the origin.

\begin{prop}
When $X = \C^n$, the quasi-isomorphism $\delta$ is defined by $\delta(\phi) = \int \phi^{n,n}$ where $\phi^{n,n}$ is the $\d^n z \d^n \zbar$-component of the expression
\beqn
\exp\left(\sum_{i=1}^n \left(\d \zbar_i \Bar{\eta}_i + \d z_i \eta_i\right)\right) j^*\phi .
\eeqn
\end{prop}

\section{Local cohomology of smooth vector fields}

In this section we prove theorem \ref{thm:smooth}.
Throughout this section $M$ is a smooth oriented manifold of dimension~$d$.
In this section, we prove the following.


\subsection{Gelfand--Kazhdan formal geometry}

Before proving this theorem we recollect some facts about the formal geometry in the style of Gelfand and Kazhdan \cite{GKF,BKformal,KapranovRW,KapranovRW}.
For more details on the specific notations used here we refer to \cite[part I]{GGW}.

Given a smooth manifold $M$ of dimension $d$, its coordinate space $M^{coor}$ is the $\infty$-dimensional (pro) manifold parametrizing jets of smooth coordinates on $M$.
A point in $M^{coor}$ consists of a point $p \in M$ together with an $\infty$-jet class of a local diffeomorphism $\phi \colon U \subset \R^d \to X$ sending a neighborhood $U$ of the origin to a neighborhood of $p$ with the property that $\phi(0) = p$.

The canonical map $M^{coor} \to M$ endows $M^{coor}$ with the structure of a principal bundle for the pro-Lie group of automorphisms of the formal $d$-disk $\op{Aut}_d$.
Additionally, there is a transitive action of the Lie algebra of vector fields $\lie{vect}(d)$ on the formal $d$-disk on $M^{coor}$.
This action defines the so-called Grothendieck connection one-form
\beqn
\omega^{coor} \in \Omega^1(M^{coor}) \otimes \lie{vect}(d)
\eeqn
which satisfies the flatness equation
\beqn\label{eqn:mc1}
\d \omega^{coor} + \frac12 [\omega^{coor},\omega^{coor}] = 0 .
\eeqn
Notice that the Lie algebra of $\op{Aut}_d$ is strictly smaller than the Lie algebra of vector fields on the formal disk, so this is not simply a flat principal $\op{Aut}_d$-bundle.
From objects, like sheaves, on the coordinate bundle $M^{coor}$ one can use this flat connection in order to obtain localized objects on the original manifold $M$.

The Jacobian of an automorphism of the $d$-disk determines a map of Lie groups $\op{Aut}_d \to GL(d)$ whose kernel is contractible.
This means that instead of working with the coordinate bundle $M^{coor}$ it suffices to work with the principal $U(d)$-bundle of frames $Fr^{\C}_M$ of the complexified tangent bundle.
The choice of a formal exponential allows one to pull back the Grothendieck connection from the full coordinate bundle to obtain a one-form $\til \omega^{coor} \in \Omega^1(Fr^{\C}_M) \otimes \lie{vect}(d)$.
Any two such choices of a formal exponential result in gauge equivalent connection one-forms.

The one-form $\til \omega^{coor}$ endows the frame bundle with the structure of a principal bundle for the pair $(\lie{vect}(d), U(d))$, see \cite{GGW}.

Now, any complex $U(d)$-representation $V$ determines a vector bundle $V_X$ on $M$ by the associated bundle construction
\beqn\label{borel}
V_M \define {\rm Fr}^\C_M \times^{U(d)} V .
\eeqn
The space $\Omega^k(M, V_M)$ of $k$-forms valued in $V_M$ is isomorphic to the space of basic $k$-forms on ${\rm Fr}_M$:
\beqn
\Omega^k({\rm Fr}_M , V)_{basic} \define \left\{\alpha \in \left(\Omega^k({\rm Fr}^\C_M) \otimes V\right)^{U(d)} \; | \; \iota_{\xi_A} \alpha = 0 \; , \; {\rm for\;all} \; A \in \lie{gl}(d) \right\}.
\eeqn
We have denoted by $\xi_A$ the vertical vector field on ${\rm Fr}^\C_M$ corresponding to~$A \in \lie{gl}(d)$. 

We consider $U(d)$-representations which have the compatible structure of a module for the Lie algebra $\lie{vect}(d)$ of vector fields on the formal $d$-disk.
Compatible means the following.
After choosing a formal coordinate, we have an embedding of Lie algebras $i : \lie{gl}(d) \to \lie{vect}(d)$, where the $d\times d$ matrix $(a_{ij})$ is realized by the vector field $\sum_{ij} a_{ij} t_i \partial_{t_j}$.
We require that the composition 
\[
\lie{gl}(d) \xto{i} \lie{vect}(d) \xto{\rho_{\lie{w}}} {\rm End}(V) 
\]
is equal to ${\rm Lie}(\rho_{U})$. 
Here $\rho_{\lie{w}}$ denotes the action of $\lie{vect}(d)$ and $\rho_{GL}$ is the original action of $U(d)$. 
Such a structure on $V$ is referred to as a {\em Harish-Chandra} module for the  pair $(\lie{vect}(d), U(d))$ in \cite{GGW}.
In what follows we will utilize a derived version of a Harish--Chandra module.
That is, we will take $V$ to be a cochain complex for which both representations $\rho_{\lie{w}}$ and $\rho_{U}$ commute with the differential.

From such data, the one-form $\omega^{coor}$ induces a connection on $V_X$ defined by
\beqn
\nabla_V \define \d + \rho_{\fw} (\omega_{\rm Groth}) .
\eeqn
The Maurer--Cartan equation (\ref{eqn:mc1}) immediately implies that this connection is flat. 
In other words, $\nabla_V$ endows sections of $V_X$ with the structure of a smooth $D_M$-module.
The resulting $D_M$-module is denoted $\op{desc}_M(V)$.
We recall that associated to any $D_M$-module $\cV$ is the dg $\Omega^\bu_M$-module
\beqn
\Omega^\bu(M, \cV) 
\eeqn
called the de Rham complex of $\cV$.

As an example, consider the algebra $\cO_d$ of functions on the formal $d$-disk; that is, the algebra $\C[[t_1,\ldots,t_d]]$ of power series in $d$ variables.
In this case, $\op{desc}_M(\cO_d)$ is the $D_M$-module of $\infty$-jets $\cJ_M = \op{jet}(M)$ of functions on $M$.
The algebra of flat sections of this $D_M$-module is simply the algebra of smooth functions $C^\infty(M)$; in other words, there is a quasi-isomorphism
\beqn
C^\infty (M) \xto{\simeq} \Omega^\bu \left(M , \op{desc}_M(\cO_d)\right) .
\eeqn

Another example is the Harish--Chandra module
\beqn
\lie{vect}(d) = \C[[t_1,\ldots,t_d]] \{\del_{t_i}\} .
\eeqn
itself, where the action of $\lie{vect}(d)$ is the adjoint action.
In this case $\op{desc}_M(\lie{vect}(d))$ is the $\infty$-jet bundle $\op{jet}(\T_M)$ associated to the tangent bundle of $M$.
The space of flat sections of this $D_M$-module is simply the space of smooth vector fields on $M$.

We will use the following basic lemma.

\begin{lem}
\label{lem:silly}
Suppose that $(V,\d)$ is a dg $(\lie{vect}(d),U(d))$-module structure wherethe $\lie{vect}(d)$-action is homotopically trivial.
Then there is a quasi-isomorphism of $\Omega^\bu_M$-modules
\beqn
\Omega^\bu(M, \op{desc}_M(V)) \simeq \left(\Omega^\bu(M) \otimes_\C V\right)_{basic} .
\eeqn
\end{lem}

\subsection{Proof of the theorem}

As recounted in equation (\ref{derham1}), there is a quasi-isomorphism of the local cohomology of a local Lie algebra $\cL$ with the shift of the de Rham complex of the $D_X$-module $C^\bu_{red}(\op{jet} (L))$.
Applied to the sheaf of smooth vector fields this reads
\beqn
\label{eqn:deRhamsmooth}
\cloc^\bu(\Vect) \; \simeq \; \Omega^\bu \bigg( M \; , \; C^\bu_{red} \left( \op{jet} (\T_M) \right) \bigg) [d] .
\eeqn

For $V$ a module for the Harish--Chandra pair $(\lie{vect}(d), U(d))$, Gelfand--Kazhdan descent along $M$ yields the $D_M$-module $\op{desc}_M(V)$. 
In the case that $V = \lie{vect}(d)$, the adjoint $\lie{vect}(d)$-module, we have seen that the $D_M$-module $\op{desc}_M(\lie{vect}(d))$ is the $D_M$-module~$\op{jet}(\T_M)$ of $\infty$-jets of the tangent bundle.

We now consider the $(\lie{vect}(d), U(d))$-module $C_{red}^\bu(\lie{vect}(d))$ where $\lie{vect}(d)$ acts in degree one through the co-adjoint action and is extended to the full dg algebra by the condition that it acts through graded derivations.
We make use of the fact that taking $\infty$-jets is a symmetric monoidal functor.
Indeed, by \cite[proposition A.2]{GGbroids}, there is a string of isomorphisms of $D_M$-modules
\beqn
\op{jet} C_{red}^\bu (\T_M) = \op{desc}_M(C_{red}^\bu(\lie{vect}(d))) \cong C_{red}^\bu(\op{desc}_M(\lie{vect}(d))) = C_{red}^\bu(\op{jet} (\T_M)).
\eeqn

To summarize, we see that the Gelfand-Kazhdan descent of the $(\lie{vect}(d), U(d))$-module $C_{red}^\bu(\lie{vect}(d))$ is the $D_M$-module $C_{red}^\bu(\op{jet} (\T_M))$.
This is precisely the $D_M$-module present in the equivalence with the complex computing the local cohomology of $\Vect$, see \eqref{eqn:deRhamsmooth}.

Combining these facts, we obtain a quasi-isomorphism of sheaves on $X$:
\beqn
\cloc^\bu(\Vect) \; \simeq \; \Omega^\bu\bigg(M , \op{desc}_M\left(C_{red}^\bu(\lie{vect}(d))\right)\bigg) [d] .
\eeqn

The interpretation via descent will allow us to simplify this Rham complex.
Suppose that $\fg$ is any Lie algebra.
Then $\fg$ acts on itself (and its dual) via the (co)adjoint action. 
This extends to an action of $\fg$ on its Chevalley--Eilenberg complex $C^\bu(\fg ; M)$, where $M$ is any $\fg$-module via the formula
\beqn\label{eqn:liederiv}
(L_x \varphi) (x_1,\ldots, x_k) = \sum_i \varphi(x_1,\ldots, \op{ad}_x(x_i), \ldots,x_k) - x \cdot \varphi(x_1,\ldots,x_k) .
\eeqn
Here, $x, x_i \in \fg$ and $\varphi$ is a $k$-cochain with values in $M$.
The symbol $\op{ad}$ denotes adjoint action, and the $\cdot$ is the $\fg$-action on $M$. 
The same formula holds for the reduced cochains.
The case of $\lie{g} = \lie{vect}(d)$ is the case we are considering here (together with the compatible $GL(d)$ action).

For any $\lie{g}$-module $N$, the Chevalley--Eilenberg complex $C^\bu(\lie{g} ; N)$ is an explicit model for the derived functor of taking invariants; therefore the $\lie{g}$-action on this complex is homotopically trivial.
Explicitly, the action by any fixed element $x\in \fg$ on $C^\bu(\fg ; N)$ can be trivialized in the following way.
Define the endomorphism $i_x$ of cohomological degree $-1$ acting on a Chevalley--Eilenberg cochain $\varphi$ via the formula
\beqn
(\iota_x \varphi) (x_1,\ldots, x_{k-1}) = \sum_i \pm \varphi (x_1,\ldots, x_i, x , x_{i+1}, \ldots, x_{k-1}) .
\eeqn
Then, from Cartan's formula
\beqn
[\d_{CE}, \iota_x] = L_x
\eeqn
it follows that $i_x$ is the desired trivialization.

Applied to the case at hand, we see that $\lie{vect}(d)$ acts homotopically trivially on $C_{red}^\bu(\lie{vect}(d))$ which implies that $\op{desc}_X(C_{red}^\bu(\lie{vect}(d))) \cong C_{red}^\bu(\op{jet}(\T_M))$ is quasi-isomorphic to the trivial $D_M$-module with fiber $C_{red}^\bu(\lie{vect}(d))$.
Equivalently, this means that the flat connection on $C_{red}^\bu(\op{jet}(\T_M))$ is gauge equivalent to the trivial connection.
From lemma \ref{lem:silly} there is a quasi-isomorphism of de Rham complexes
\beqn
(\Omega^\bu\bigg(M ,  C_{red}^\bu\left(\op{jet}(\T_M)\right)\bigg) \simeq \left(\Omega^\bu_M \otimes_\C C^\bu_{red}(\lie{vect}(d))\right)_{basic} .
\eeqn

Suppose that $Y$ is a smooth manifold equipped with a $U(n)$-action.
Then, basic forms provide a model for the de Rham cohomology of $Fr^\C_M \times^{U(n)} Y$:
\beqn\label{eqn:basicsmooth}
\Omega^\bu(Fr^\C_M \times^{U(n)} Y) \cong (\Omega^\bu(Fr^\C_M) \otimes \Omega^\bu(Y))_{basic} .
\eeqn
Now, we recall a characterization of the Lie algebra cohomology of $\lie{vect}(d)$.
Let $X(d)$ be the restriction of the universal $U(n)$ bundle $EU(n) \to BU(n)$ over the $2d$-skeleton of $BU(d)$.
A famous theorem of Gelfand and Fuks \cite{GF1} states a quasi-isomorphism $C_{Lie}^\bu(\lie{vect}(d)) \simeq \Omega^\bu(X(d))$.
Applying this to the case $Y=X(d)$ in equation \eqref{eqn:basicsmooth} recovers the desired result.


\section{Local cohomology of holomorphic vector fields}
\label{sec:holomorphic}

Next we turn to the cohomology of the local Lie algebra of holomorphic vector fields on a complex manifold.
In this section, we will prove theorem \ref{thm:hol}.

\subsection{Holomorphic descent}

The method of proof is completely parallel to the smooth case above.
We remark on some important features present in the holomorphic case, we refer to \cite{GGW} for more details.

For $X$ a complex manifold, $X^{coor}$ will now denote the \textit{holomorphic} coordinate bundle which consists of pairs $(x,\phi)$ where $x \in X$ and $\phi$ is the $\infty$-jet of a local holomorphic diffeomorphism $\phi  \colon U \subset \C^n \to X$ with the property that $\phi(0) = x$.
The bundle $X^{coor} \to X$ is a principal $\op{Aut}_n$-bundle which is equipped with a transitive action of $\op{vect}(n)$ as before.
The Grothendieck connection is now a holomorphic one-form $\omega^{coor} \in \Omega^{1,hol}(X^{coor}) \otimes \lie{vect}(n)$ which satisfies 
\beqn
\del \omega^{coor} + \frac12 [\omega^{coor},\omega^{coor}] = 0.
\eeqn

Given a $C^\infty$ section of $X^{coor} \to X^{coor} / U(n)$ (a formal exponential) we obtain a $\lie{vect}(n)$-valued one-form $\til\omega^{coor} \in \Omega^1(\op{Fr}_X) \otimes \lie{vect}(n)$ on the $U(n)$-frame bundle of $X$.
This endows the frame bundle with the structure of a principal bundle for the pair $(\lie{vect}(n),U(n))$.
In coordinates one simply has $\omega^{coor} = \d z_i \otimes \frac{\del}{\del x_i}$ where $z$ is a coordinate on $X$ and $x$ is a formal coordinate.

Fix a $(\lie{vect}(n),U(n))$-module $(V,\rho_{\lie{w}})$ and let $V_X$ be the associated bundle $V_X = \op{Fr}_X \times^{U(n)} V$.
We have the flat connection
\beqn
\nabla_V^{flat} = \d + \rho_{\fw} (\omega^{coor}) 
\eeqn 
on $V_X$.
Moreover, since $V_X$ is a holomorphic bundle over $X$, there is a quasi-isomorphism of sheaves
\beqn
\left(\Omega^{\bu, hol} (X , V_X) \; , \; \nabla_V \right) \xto{\simeq}\left(\Omega^{\bu, \bu} (X , V_X) \; , \; \nabla_V^{flat} \right) 
\eeqn 
where $\nabla_V = \dbar -  \nabla_V^{flat}$.
Here $\dbar$ is the $\dbar$-operator for $V_X$.
The operator $\nabla_V$ endows $\cV^{hol}_X = \Gamma^{hol}(X,V_X)$ with the structure of a $D^{hol}_X$-module, where $D^{hol}_X$ is the algebra of holomorphic differential operators.
We refer to this $D^{hol}_X$-module as $\op{desc}_X^{hol}(V)$.


\subsection{Proof of the theorem}

As recounted in equation (\ref{derham1}), there is a quasi-isomorphism of the local cohomology of a local Lie algebra $\cL$ with the de Rham complex of the $D_X$-module $C^\bu_{red} (\op{jet} (L))$.
Applied to the local Lie algebra $\cT$ on the complex $n$-fold $X$, this reads:
\beqn
\cloc^\bu(\cT) \; \simeq \; \Omega^\bu \bigg( X \; , \; C^\bu_{red} \left( \op{jet} (\cT) \right) \bigg) [2n] .
\eeqn





As sheaves, we know $\cT$ is a resolution for the sheaf of holomorphic vector fields on the complex manifold $X$. 
Similarly, there is a quasi-isomorphism of (smooth) $D_X$-modules $\op{jet}(\cT) \simeq \op{jet}^{hol} (\T^{1,0})$, where $\op{jet}^{hol}(\T^{1,0})$ denotes the holomorphic bundle of holomorphic $\infty$-jets of the holomorphic tangent bundle. 
It follows that there is a quasi-isomorphism of de Rham complexes
\beqn\label{smoothhol}
\Omega^\bu \bigg( X \; , \; C^\bu_{red} \left( \op{jet} (\cT) \right) \bigg) \; \simeq \; \Omega^\bu \bigg(X \; , \; C_{red}^\bu\left(\op{jet}^{hol} (\T^{1,0}) \right) \bigg) .
\eeqn
On the right-hand side we emphasize that we take {\em holomorphic} jets.

If $V$ is a module for the pair $(\lie{vect}(n), U(n))$, then holomorphic Gelfand-Kazhdan descent along the complex manifold $X$ yields a $D^{hol}_X$-module $\op{desc}^{hol}_X(\cV)$. 
In the case that $V$ is the adjoint module $\lie{vect}(n)$, the $D^{hol}_X$-module $\op{desc}_X(\cV)$ is equivalent to the $D^{hol}_X$-module $\op{jet}^{hol} (\T^{1,0})$. 

We now consider the $(\lie{vect}(n), U(n))$-module $C_{red}^\bu(\lie{vect}(n))$. 
By \cite[proposition A.2]{GGbroids}, there is a string of isomorphisms of $D^{hol}_X$-modules
\beqn
\op{jet}^{hol} C_{red}^\bu (\T^{1,0}_X) = \op{desc}(C_{red}^\bu(\lie{vect}(n))) \cong C_{red}^\bu(\op{desc}(\lie{vect}(n))) = C_{red}^\bu({\rm J}^{hol} \T^{1,0}_X),
\eeqn
where ${\rm J}^{hol} \T^{1,0}_X$ is the bundle of holomorphic $\infty$-jets of $\T^{1,0}_X$.

To summarize, we see that the Gelfand-Kazhdan descent of the $(\lie{vect}(n), U(n))$-module $C_{red}^\bu(\lie{vect}(n))$ is equal to the $D_X$-module $C_{red}^\bu(\op{jet}^{hol} \T^{1,0})$.
From (\ref{smoothhol}), this is precisely the $D_X$-module present in the definition of the local cohomology of $\cT$.

Combining these facts, we obtain a quasi-isomorphism of sheaves on $X$:
\beqn
\cloc^\bu(\cT) \; \simeq \; \Omega^\bu\bigg(X , \op{desc}_X\left(C_{red}^\bu(\lie{vect}(n))\right)\bigg) [2n] .
\eeqn
Following the exact same logic as in the smooth case yields the desired result.

\section{Cohomology of formal vector fields}
\label{sec:gf}

We have seen that the local cohomology of smooth or holomorphic vector fields is largely determined by the cohomology of \textit{formal} vector fields.
In this section we collect some facts about the cohomology of formal vector fields which will allow us to construct explicit models for representatives of local cohomology classes.

As before, $\lie{vect}(n)$ is the Lie algebra of formal vector fields on the $n$-disk, and let $\Omega^p$ be the $\lie{vect}(n)$-module of formal differential $p$-forms on the $n$-disk.
Thus $\cO = \Omega^0 = \C[[x_1,\ldots,x_n]]$ and $\Omega^p$ is freely generated over $\cO$ by formal symbols $\d x_{i_1} \wedge \cdots \wedge \d x_{i_p}$.
The formal de Rham differential $\d_{dR} \colon \Omega^p \to \Omega^{p+1}$ is defined in the obvious way as is the interior product $\iota_X \colon \Omega^p \to \Omega^{p-1}$, where $X \in \lie{vect}(n)$. 
The $\lie{vect}(n)$-module structure on $\Omega^p$ is through the Lie derivative $L_X \colon \Omega^p \to \Omega^p$ where $L_X = \d_{dR} \iota_X + \iota_X \d_{dR}$.

We recall Gelfand--Fuks cohomology with coefficients in formal differential forms.
For each $p$ we define the cochain complex
\beqn
C^\bu(\lie{vect}(n) ; \Omega^p)
\eeqn
which in degree $q$ consists of continuous linear maps $\varphi \colon \wedge^q \lie{vect}(n) \to \Omega^p$.
The differential is $\d_{CE}^p \colon C^q \to C^{q+1}$ is defined by
\begin{multline}\label{eqn:CEp}
(\d_{CE}^{p} \varphi)(X_0,\ldots,X_q) = \sum_{i<j} (-1)^{i+j} \varphi([X_i,X_j], X_1,\ldots,\Hat{X}_i, \ldots, \Hat{X}_j, \ldots, X_q) \\ + \sum_i (-1)^{i+1} L_{X_i} \varphi(X_0,\ldots,\Hat{X}_i, \ldots,X_q) .
\end{multline}
We write 
\beqn
\d_{CE}^p = \d_{\lie{w}} \otimes \id_{\Omega^p} - L
\eeqn
where the first summand is described in the first line above and the second summand is described by the (negative of) the second line.

To obtain expressions for representatives of cohomology classes in $H^\bu(\lie{vect}(n);\C)$ we will use the double complex
\beqn
C^\bu \left(\lie{vect}(n) ; \Omega^\bu \right)
\eeqn
where the horizontal differential is $\d_{CE}$ and the vertical differential is $\d_{dR}$. 
The totalized differential on the right hand side is a sum of two terms
\beqn
\d = \d_{CE} + \d_{dR} .
\eeqn
The inclusion of constant functions $\C \hookrightarrow \Omega^\bu$ is $\lie{vect}(n)$-equivariant, hence induces a map
\beqn
C^\bu(\lie{vect}(n);\C) \to \op{Tot} \, C^\bu \left(\lie{vect}(n) ; \Omega^\bu \right) .
\eeqn
By the formal Poincar\'{e} lemma, this map is a quasi-isomorphism.
The main goal of the remainder of this section is to construct a quasi-inverse to this map.

\subsection{Cohomology with coefficients in differential forms}

We summarize a description of the bigraded cohomology ring
\beqn
H^\bu(\lie{vect}(n) ; \Omega^\bu) = \oplus_q \oplus_p H^q (\lie{vect}(n) ; \Omega^p) .
\eeqn
in terms of formal characteristic classes first given in \cite{GFcoefficients}.

For each $k \geq 0$ there is a decreasing sequence of subalgebras
\beqn
\lie{vect}(n) \supset \lie{vect}(n)_{0} \supset \lie{vect}(n)_{1} \supset \cdots
\eeqn
where
\beqn
\lie{vect}(n)_{k} = \{\sum_i f_i \del_i \; | \; f_i \in \lie{m}^{k+1} \C[[z_1,\ldots,z_n]] \} \subset \lie{vect}(n) ,
\eeqn
is the subalgebra of vector fields which vanish up to order $k$.

The formal Jacobian is the linear map
\beqn
J \colon \lie{vect}(n) \to \lie{gl}(n)[[z_1,\ldots,z_n]]
\eeqn
defined by sending a vector field $X = \sum_i f_i \del_i$ to the matrix of power series $JX$ whose $ij$ component is $\del_i f_j$.
There is map of Lie algebras
\beqn
J(0) \colon \lie{vect}(n)_{0} \to \lie{gl}(n)
\eeqn
defined by evaluating the formal Jacobian at zero $X \mapsto JX(0)$.
This defines an isomorphism $\lie{vect}(n)_{0} \slash \lie{vect}(n)_{1} \simeq \lie{gl}(n)$ with inverse that sends a matrix $A$ to the vector field $\sum A_{ij} z_i \del_j$.
The map $J(0)$ allows us to restrict any $\lie{gl}(n)$-representation to one for $\lie{vect}(n)_{0}$.

Given any $\lie{vect}(n)_{0}$-module $M$ we induce along $\lie{vect}(n)_{0} \subset \lie{vect}(n)$ to obtain the $\lie{vect}(n)$-module
\beqn
\til M = \Hom_{U \lie{w}_{n,0}} (U \lie{vect}(n), M) .
\eeqn
In the case that $M = \wedge^p (\C^n)^*$ the corresponding induced module is $\til M = \Omega^p$.
In particular we see that there is an isomorphism of cochain complexes
\beqn
C^\bu(\lie{vect}(n) ; \Omega^n) \simeq C^\bu(\lie{vect}(n)_{0} ; \wedge^p (\C^n)^*) .
\eeqn 
In fact, this induces an isomorphism of bigraded rings $C^\bu(\lie{vect}(n) ; \Omega^\bu) \simeq C^\bu(\lie{vect}(n)_{0}; \wedge^\bu (\C^n)^*)$.

The embedding $\lie{gl}(n) \hookrightarrow \lie{vect}(n)_{0}$ induces an isomorphism
\beqn
C^\bu(\lie{vect}(n) ; \cO) \simeq C^\bu (\lie{vect}(n)_{0} ; \C) \simeq C^\bu(\lie{gl}(n) ; \C) .
\eeqn
And so the cohomology with coefficients in $\cO$ is an exterior algebra
\beqn
H^\bu(\lie{vect}(n) ; \cO) \simeq \C[a_1,\ldots,a_n]
\eeqn
where $a_i$ is of degree $2i-1$.
Explicitly, a representative for $a_i$ is
\beqn
a_i \colon X_0,\ldots,X_{2n} \mapsto \op{Tr}(JX_0 \cdots J X_{2n}) .
\eeqn

\begin{thm}[\cite{GFcoefficients}]\label{thm:coefficients}
The bigraded ring
\beqn
H^\bu(\lie{vect}(n) ; \Omega^\bu)
\eeqn
is a graded polynomial algebra on generators $a_1,\ldots,a_n$ where $a_i$ is bidegree $(0,2i-1)$ and generators $\tau_1,\ldots,\tau_n$ where $\tau_j$ is bidegee $(j,j)$ modulo the relation
\beqn
\tau_1^{\ell_1} \tau_2^{\ell_2} \cdots \tau_n^{\ell_n} = 0, \quad \text{for} \quad \ell_1 + 2 \ell_2 + \cdots n \ell_n > n .
\eeqn
Explicit cochain representatives are as follows.
\begin{itemize}
\item A representative for $a_i$ is
\beqn
a_i \colon X_0,\ldots,X_{2n} \mapsto \op{Tr}(JX_0 \cdots J X_{2n}) ,
\eeqn
where $JX_i \in \lie{gl}(n)[[z_1,\ldots,z_n]]$ is the formal Jacobian.
\item A representative for $\tau_i$ is 
\beqn
\tau_i \colon X_1,\ldots,X_n \mapsto \op{Tr}( \d J X_1 \wedge \cdots \wedge \d J X_n ) ,
\eeqn
where $\d J X_i \in \lie{gl}(n) \otimes \Omega^1$ denotes the matrix of formal one-forms.
\end{itemize}
\end{thm}

Recall the space $X(n)$ defined as the restriction of the universal $U(n)$-bundle to the $2n$-skeleton of $Gr(n,\infty)$.
It's cohomology is isomorphic to the cohomology of $\lie{vect}(n)$.
There is a minimal model for the de Rham cohomology of $X(n)$ which takes the form
\beqn
\left(\C[\xi_1,\ldots, \xi_n, c_1,\ldots,c_n] \slash (c_1^{\ell_1} c_2^{\ell_2} \cdots c_n^{\ell_n}) \; , \; \d = c_i \frac{\partial}{\partial \xi_i} \right),
\eeqn
where the relation is imposed for $\ell_1+2 \ell_2 + \cdots + n \ell_n > n$.
Here $\xi_i$ sits in degree $2i-1$ and $c_i$ sits in degree $2i$.
We recognize the $c_i$'s as the universal Chern classes in $n$-dimensions and $H^\bu(U(n)) \simeq \C[\xi_1,\ldots,\xi_n]$. 
The evident relationship to the formal Hodge-to-de Rham spectral sequence on generators is $\xi_i \leftrightarrow a_i$ and $c_i \leftrightarrow \tau_i$.
\subsection{Formal descent}

Define an operator 
\beqn
\iota \colon C^q(\lie{vect}(n) ; \Omega^p) \to C^{q+1} (\lie{vect}(n) ; \Omega^{p-1}) ,
\eeqn
which takes a cochain $\varphi \colon \wedge^q \lie{vect}(n) \to \Omega^p$ to the cochain
\beqn
(\iota \varphi)(X_1,\ldots,X_q) = \sum_k (-1)^{k+1} \iota_{X_k} \varphi(X_1,\ldots,\Hat{X}_k, \ldots,X_q) ,
\eeqn
where, on the right hand side, $\iota_X \colon \Omega^p \to \Omega^{p-1}$ is the contraction with respect to~$X \in \lie{vect}(n)$.

Next we define what is meant by evaluation at zero.
For $\alpha \in C^q(\lie{vect}(n);\Omega^p)$, define $\alpha|_0 \in C^q(\lie{vect}(n);\C)$ as follows.
If $p > 0$ then $\alpha|_0 = 0$ and if $p = 0$ then
\beqn
\alpha|_0 (X_1,\ldots,X_q) = \alpha(X_1,\ldots,X_q)(0) \in \C .
\eeqn

\begin{prop}\label{prop:quasi}
The map
\beqn
\Phi \colon \op{Tot} C^\bu(\lie{vect}(n);\Omega^\bu) \to C^\bu(\lie{vect}(n);\C) 
\eeqn
defined by
\beqn
\Phi (\alpha) = e^{\iota} (\alpha) |_{0} 
\eeqn
is a quasi-isomorphism of commutative dg algebras.
\end{prop}
\begin{proof}
In this proof we will use the notations
\beqn
C^q(\Omega^p) \define C^q(\lie{vect}(n); \Omega^p) , \quad C^q \define C^q(\lie{vect}(n); \C) ,
\eeqn
and $D^p = \d_{CE}^p + \d_{dR}$ the total differential acting on $C^q(\Omega^p)$.
For $\alpha \in C^q(\Omega^\bu)$ we denote $\alpha|_{\cO} \in C^q(\cO)$ the projection onto the function component of $\alpha$.
Notice that $\alpha|_0 = (\alpha|_{\cO})|_0$.

First, we show that $\Phi$ is a map of graded algebras.
Indeed, suppose that $\alpha \in C^q(\Omega^p), \alpha' \in C^{q'}(\Omega^{p'})$.
Then
\beqn
e^{\iota} \alpha|_{\cO} = \frac{1}{p!} \iota^p \alpha
\eeqn
and similarly for $\alpha'$.
Explicitly, this is the cochain
\beqn
e^{\iota} \alpha|_{\cO} (X_1,\ldots,X_{p+q}) = \sum_{I,J} \iota_{X_{i_1}} \cdots \iota_{X_{i_p}} \alpha(X_{j_1},\cdots,X_{j_q}) 
\eeqn
where the sum is over increasing multi-indices $I = (i_1< \cdots <i_p)$, $J=(j_1<\cdots<j_q)$ with the property that $I \cup J = \{1,\ldots,p+q\}$.
Now,
\beqn
e^{\iota} (\alpha \alpha')|_{\cO} = \frac{1}{(p+p')!} \iota^{p+p'} (\alpha \alpha') .
\eeqn
Explicitly, this is the cochain
\begin{multline}
e^{\iota} (\alpha \alpha')|_{\cO} (X_1,\ldots,X_{p+p'+q+q'}) = \sum_{A,B} \iota_{X_{a_1}} \cdots \iota_{X_{a_{p+p'}}} (\alpha \alpha')(X_{b_1},\cdots,X_{b_{q+q'}}) \\ 
= \sum_{I,J,I',J'} \iota_{X_{i_1}} \cdots \iota_{X_{i_{p}}} \alpha (X_{j_1},\cdots,X_{j_{q}}) \iota_{X_{i'_1}} \cdots \iota_{X_{i'_{p'}}}  \alpha'(X_{j'_1},\ldots,X_{j'_{q'}})
\end{multline}
where the first sum is over increasing multi-indices $A,B$ of length $p+p',q+q'$ respectively.
The second sum is over multi-indices $I,J$ of length $p,q$ respectively and $I',J'$ of length $p',q'$ respectively.
The second line is precisely $(e^{\iota} \alpha|_{\cO} e^{\iota} \alpha'|_{\cO})(X_1,\ldots,X_{p+q+p'+q'})$ as desired.

As a graded algebra $C^\bu(\Omega^\bu)$ is generated by elements in $C^\bu(\cO)$ and $C^0(\Omega^\bu)$.
Thus, to show that $\Phi$ is a cochain map it suffices to show that $\Phi(D \psi) = \d_{\lie{w}} \Phi(\psi)$ for $\psi$ in each of these pieces.

Suppose first that $\varphi \in C^q(\lie{vect}(n) ; \cO)$.
On one hand $\Phi(\varphi) = \varphi|_0$ and so
\beqn
\d_{\lie{w}} \Phi(\varphi) = (\d_{\lie{w}} \otimes \id_{\cO}) \varphi |_0
\eeqn
where $\d_{\lie{w}} \otimes \id_{\cO}$ is the operator described in the first line in equation \eqref{eqn:CEp}.
On the other hand
\begin{align*}
\Phi(D^0 \varphi) & = (\d_{CE}^\cO \varphi + \iota \d_{dR} \varphi)|_0 \\ & = (\d_{CE}^\cO \varphi - L \varphi)|_0 ,
\end{align*}
where $L$ is the operator described in the second line of \eqref{eqn:CEp}.
Since (by definition) $\d_{CE}^\cO = \d_{\lie{w}} \otimes \id_{\cO} - L$, we have shown that $\Phi(D^0 \varphi) = \d_{\lie{w}} \Phi(\varphi)$.

Now, suppose that $\alpha \in C^0(\Omega^p) = \Omega^p$.
Then $\Phi(\alpha) = \frac{1}{p!} \iota^p \alpha|_0 \in C^p$ is the $p$-cochain
\beqn
\Phi(\alpha)(X_0,\ldots,X_{p-1}) = \iota_{X_0} \cdots \iota_{X_{p-1}} \alpha|_0 .
\eeqn
Thus $\d_{\lie{w}} \Phi (\alpha)\in C^{p+1}$ is the $(p+1)$-cochain
\beqn\label{eqn:wdifferential}
\d_{\lie{w}} \Phi (\alpha) (X_0,\ldots,X_p) = \sum_{i<j} (-1)^{i+j} \iota_{[X_i,X_j]}\iota_{X_0} \cdots \Hat{\iota_{X_i}} \cdots \Hat{\iota_{X_j}} \cdots \iota_{X_{p}} \alpha|_0 .
\eeqn

On the other hand, we have
\begin{align*}
\Phi(D^p \alpha) & = \Phi(\d_{CE}^p \alpha) + \Phi(\d_{dR} \alpha) \\
& = \frac{1}{p!} \iota^p \d_{CE}^p \alpha|_0 + \frac{1}{(p+1)!} \iota^{p+1} \d_{dR} \alpha|_0 .
\end{align*}
Now, $\d_{CE}^p \alpha \in C^{1}(\Omega^p)$ is the $\Omega^p$-valued $1$-cochain 
\beqn
(\d_{CE}^p \alpha)(X) = - L_X \alpha .
\eeqn
Thus, $\frac{1}{p!} \iota^p (\d_{CE}^p \alpha)$ is the $(p+1)$-cochain
\beqn\label{eqn:CEp}
\frac{1}{p!} \iota^p (\d_{CE}^p \alpha)(X_0,\ldots,X_p) = \sum_i (-1)^{i} \iota_{X_0} \cdots \Hat{\iota_{X_i}} \cdots \iota_{X_{p-1}} L_{X_i} \alpha .
\eeqn

Next, we read off the $(p+1)$-cochain $\frac{1}{(p+1)!} \iota^{p+1} \d_{dR} \alpha$ as
\beqn
\frac{1}{(p+1)!} \iota^{p+1} \d_{dR} \alpha (X_0,\ldots,X_p) = \iota_{X_0}\cdots \iota_{X_p} \d_{dR} \alpha .
\eeqn
Via Cartan's formula $L_X = \d_{dR} \iota_X + \iota_X \d_{dR}$ we can write this as
\beqn
\iota_{X_0}\cdots \iota_{X_{p-1}} (L_{X_p} - \d_{dR} \iota_{X_p}) \alpha =  \iota_{X_0}\cdots \iota_{X_{p-1}} L_{X_p} \alpha - \iota_{X_0}\cdots \iota_{X_{p-1}}\d_{dR} \iota_{X_p} \alpha .
\eeqn
Iterating, we can use Cartan's formula to remove all explicit appearances of the de Rham differential
\beqn
\frac{1}{(p+1)!} \iota^{p+1} \d_{dR} \alpha (X_0,\ldots,X_p) = \sum_{i} \pm \iota_{X_0} \cdots \iota_{X_{i-1}} L_{X_i} \iota_{X_{i+1}} \cdots \iota_{X_p} \alpha 
\eeqn
Here, we have use the fact that since $\alpha$ is a $p$-form that $\d_{dR} \iota_{X_0} \cdots \iota_{X_p} \alpha = 0$.
We now use the identity $[\iota_X,L_Y] = \iota_{[X,Y]} = [L_X, \iota_Y]$ to place all Lie derivatives on the right.
This becomes
\begin{multline}
\frac{1}{(p+1)!} \iota^{p+1} \d_{dR} \alpha (X_0,\ldots,X_p) = \sum_i (-1)^{i+1} \iota_{X_0} \cdots \Hat{\iota_{X_i}} \cdots \iota_{X_{p-1}} L_{X_i} \alpha \\
+ \sum_{i<j} (-1)^{i+j} \iota_{[X_i,X_j]}\iota_{X_0} \cdots \Hat{\iota_{X_i}} \cdots \Hat{\iota_{X_j}} \cdots \iota_{X_{p}} \alpha .
\end{multline}
The first line cancels with \eqref{eqn:CEp} and the second line, when evaluated at zero, matches \eqref{eqn:wdifferential}.
We conclude that $\Phi$ is a cochain map.

By the formal Poincar\'{e} lemma, we know that the inclusion
\beqn
i\colon \left(C^\bu, \d_{\lie{w}}\right) \hookrightarrow \left(C^\bu(\Omega^\bu), D\right)
\eeqn
induced by the inclusion $\C \hookrightarrow \cO$ of constant functions is a quasi-isomorphism.
It is clear that $\Phi \circ i = \id_{C^\bu}$. 
To conclude the proof, we will construct a homotopy $i \circ \Phi \simeq \id_{C^\bu(\Omega^\bu)}$.
Consider the commutative dg algebra
\beqn
\Omega^\bu([0,1]) \Hat{\otimes} C^\bu(\Omega^\bu) ,
\eeqn
which is equipped with the differential $\d_t + D = \d_t + \d_{CE}^{\Omega^\bu} + \d_{dR}$ where $\d_t$ denotes the de Rham differential on the interval $[0,1]$.
Fiberwise integration is a linear map
\beqn
\int_{[0,1]} \colon \Omega^\bu([0,1]) \Hat{\otimes} C^\bu(\Omega^\bu) \to C^\bu(\Omega^\bu)[-1] .
\eeqn
Explicitly, $\int_{[0,1]} f(t,x) \alpha (x) = 0$ and $\int_{[0,1]} f(t,x) \d t \alpha = \left(\int_0^1 f(t,x) \d t\right) \alpha(x)$, for $f \in C^\infty([0,1])$, $\alpha \in C^\bu(\Omega^\bu)$.
Now, consider $\alpha \in C^q (\Omega^p)$, which we will write as
\beqn
\alpha(X_1,\ldots,X_q) = \sum_I \alpha(X_q,\ldots,X_q)(x)_I \d x^I
\eeqn
where $I$ is a multi-index.
Define 
\beqn
h^* \alpha \in \Omega^1([0,1]) \Hat{\otimes} C^q(\Omega^{p-1})
\eeqn
as
\beqn
(h^* \alpha)(X_1,\ldots,X_n)(t,x) = \sum_I \alpha(X_q,\ldots,X_q)(tx)_I \d(tx^I) .
\eeqn
Define $\Psi (\alpha) = \int_{[0,1]} \circ h^*(\alpha)$.

Observe that $\int_{[0,1]}$ anticommutes with the Chevalley--Eilenberg differential.
Thus by the same argument as in the proof of the Poincar\'{e} lemma we see that 
\beqn
i \circ \Phi - \id = D \circ \Psi + \Psi \circ D .
\eeqn
This shows that $\Phi$ is a homotopy equivalence, as desired.

\end{proof}

\subsection{Examples}

Using the above results, we will work out representatives of Gelfand--Fuks cohomology classes in some low-dimensional examples.

\subsubsection{}
Let's consider the one-dimensional case.
It is well-known that $H^\bu(\lie{w}_1;\C)$ is concentrated in degree zero and three, and $H^3(\lie{w}_1;\C)$ is one-dimensional.
The spectral sequence 
\beqn
C^\bu(\lie{w}_1;\Omega^\bu) \implies \C^\bu(\lie{w}_1;\C)
\eeqn
collapses at the $E_2$-page and the only remaining generator is $a_1 \tau_1$ which is of bidegree $(1,2)$ which we view as an element
\beqn
a_1 \tau_1 \in C^2 (\lie{w}_1 ; \Omega^1) .
\eeqn
The image of this element under $\Phi$ will thus give a generator for the cohomology $H^3 (\lie{w}_1;\C)$.
Note that $\Phi(a_1 \tau_1) = -a_1 (\iota \tau_1)|_0$ and that
\beqn
(\iota \tau_1)(f\del_x, g\del_x) = f \del_x^2 g - g \del_x^2 f \in \cO ,
\eeqn
where $f,g \in \cO$.
Thus
\beqn\label{eqn:GF1d}
\Phi(a_1 \tau_1)(f\del_x, g\del_x, h \del_x) = \op{det} \begin{pmatrix} f & g & h \\ \del_x f & \del_x g & \del_x h \\ \del_x^2 f & \del_x^2 g & \del_x^2 h \end{pmatrix} (0) ,
\eeqn
recovering the well-known formula for this cocycle, for example see \cite{Fuks}.

\subsubsection{}
Let's consider the two-dimensional case.
As before, we use the results above together with the spectral sequence
\beqn
E_1 = H^\bu(\lie{w}_2 ; \Omega^\bu) \implies H^\bu(\lie{w}_2)
\eeqn
to construct representatives for the cohomology classes in $H^5(\lie{w}_2)$, which we know is a two-dimensional vector space.
The $E_1$ page is the free commutative bigraded algebra on generators $(a_1,a_3,\tau_1,\tau_2)$ of bidegrees $((0,1), (0,3),(1,1), (2,2))$ subject to the relations
\beqn
\tau_1^3 = \tau_1\tau_2 = \tau_2^2 = 0 .
\eeqn
The differential on the $E_1$ page is $\d \tau_1 = a_1$.
On the $E_2$ page the differential is determined by $\d \tau_2 = a_2$.
From this, we see that the spectral sequence collapses at the $E_3$ page.
The cohomology we are interested lives purely in bidegree $(2,3)$ on this page and is represented by the elements $a_1 \tau_1^2, a_2 \tau_2^2$.

From this description it is immediate to see that the elements $a_1 \tau_1^2 , a_1 \tau_2 \in H^3(\lie{w}_2 ; \Omega^2)$ survive to the $E_\infty$-page.
Now $\Phi(a_1\tau_2) = -\frac12 a_1 \iota^2 \tau_2$ and
\begin{multline}
\frac12 \iota^2 \tau_2 (X_1,X_2,X_3,X_4) = \iota_{X_1} \iota_{X_2} \op{Tr}(\d J X_3 \d J X_4) - \iota_{X_1} \iota_{X_3} \op{Tr}(\d J X_2 \d J X_4) + \iota_{X_1} \iota_{X_4} \op{Tr}(\d JX_2 \d JX_3) \\
+\iota_{X_2} \iota_{X_3} \op{Tr}(\d J X_1 \d J X_4) - \iota_{X_2} \iota_{X_4} \op{Tr}(\d J X_1 \d J X_3) + \iota_{X_3} \iota_{X_4} \op{Tr}(\d JX_1 \d JX_2) .
\end{multline}
Thus
\beqn
\Phi(a_1 \tau_2)(X_0,\ldots,X_4) = \op{Tr}(JX_0) (0) \iota_{X_1} \iota_{X_2} \op{Tr}(\d J X_3 \d J X_4)(0) + \cdots
\eeqn
where the $\cdots$ denote terms which make the cocycle manifestly antisymmetric.
Similarly
\beqn
\Phi(a_1 \tau_1^2)(X_0,\ldots,X_4) = \op{Tr}(JX_0)(0) \iota_{X_1} \iota_{X_2} \op{Tr}(\d J X_3)(0) \op{Tr}(\d J X_4)(0) + \cdots
\eeqn

\section{Descent equations and local representatives}\label{sec:descent}

The goal in this section is to describe an explicit representatives for local cohomology classes following theorem \ref{thm:hol} in the affine case $X = \C^n$.
We will focus on holomorphic vector fields, but in the last part of this section we comment on results for smooth vector fields.

\subsection{Descent, two ways}
Recall that the global sections on $\C^n$ of the local dg Lie algebra $\cT$ is the dg Lie algebra $\cT(\C^n) = \Omega^{0,\bu}(\C^n, \T^{1,0})$.
Dolbeault's theorem implies that there is a quasi-isomorphism of dg Lie algebras
\beqn
\cT(\C^n) \xto{\simeq}  H^0(\cT(\C^n)) = \Vect^{hol}(\C^n) ,
\eeqn
where $\Vect^{hol}(\C^n)$ is the Lie algebra of holomorphic vector fields on $\C^n$.
Taylor expansion at $0 \in \C^n$ defines a map of Lie algebras from holomorphic vector fields to formal vector fields:
\beqn
j_0^\infty \colon \Vect^{hol} (\C^n)  \to \lie{vect}(n)
\eeqn
which takes the Taylor expansion of a holomorphic vector field at $0 \in \C^n$. 
We denote the composition $j \define j_0^\infty \circ p : \cT(\C^n) \to \lie{vect}(n)$. 
The map $j$ defines a map on the continuous Chevalley--Eilenberg cochain complexes
\beqn
j^* \colon C^\bu(\lie{vect}(n)) \to C^\bu\left(\cT(\C^n))\right) .
\eeqn

Associated to the local Lie algebra $\cT$ we have the $\infty$-jet bundle $J (\cT)$. 
We denote the dg Lie algebra of global sections of this jet bundle by $\op{jet} (\cT) (\C^n)$.
By construction, we note that the map $j^*$ factors through the embedding of cochain complexes
\beqn
C^\bu\bigg( \op{jet} (\cT) (\C^n) \bigg) \hookrightarrow C^\bu \left( \cT(\C^n) \right) .
\eeqn
So, we obtain for each $\phi \in C^\bu(\lie{vect}(n))$ a section $j^* \phi$ of the vector bundle $C^\bu\left(\op{jet}(\cT)\right)$. 

\begin{eg}
Suppose $n=1$ and consider the $1$-cochain $\phi : f(x) \frac{\partial}{\partial x} \mapsto f'(0)$ of $\fw_1$. 
The value of the section $j^* \phi$ at the point $z_0 \in \C$ is the cochain for $\cT = \Omega^{0,\bu}(\C, T_\C)$ defined by
\beqn
a(z,\zbar) \frac{\partial}{\partial z} + b(z,\zbar) \d \zbar \frac{\partial}{\partial z} \mapsto \frac{\partial}{\partial z} a(z,\zbar) |_{z = z_0} .
\eeqn
\end{eg}

On any manifold, we have seen that $C^\bu\left(\op{jet}(\cT)\right)$ is commutative dg algebra in the category of $D$-modules.
On $\C^n$ we consider the associated de Rham complex
\beqn\label{eqn:deRhamdescent}
\Omega^\bu \bigg( \C^n \; , \; C^\bu \left( \op{jet} (\cT) \right) \bigg)
\eeqn
Recall that up to a shift (and upon taking reduced cochains) this complex is quasi-isomorphic to the complex which computes the local cohomology of $\cT$. 

Via the map $j$, a cochain $\phi \in C^\bu(\lie{vect}(n))$ determines a zero form in this de Rham complex 
\beqn
\phi^0 \define j^* \phi \in \Omega^0\left(\C^n , C^\bu(\op{jet}(\cT)) \right) .
\eeqn
The section $\phi^0$ is not flat, but we have the following. 

\begin{prop}
\label{prop:gfdescent}
Suppose $\phi \in C^\bu(\lie{vect}(n))$ and let $\phi^0 = j^* \phi$. 
Then, there exists $\phi^{i,j} \in \Omega^{i,j} (\C^n , C^\bu \left( \op{jet} (\cT) \right))$, $1 \leq i,j \leq n$ such that the element 
\beqn
\Phi \define \sum_{i,j} \phi^{i,j} 
\eeqn
satisfies the equation $(\del + \dbar + \d_{\cT}) \Phi = 0$. 
\end{prop}

Using the Hodge decomposition of the Rham differential $\d_{\rm dR} = \dbar + \partial$, we we will show that the elements $\phi^{i,j}$ satisfy a pair of descent equations:
\begin{itemize}
\item Holomorphic descent equations:
\beqn\label{eqn:holdescent}
\dbar \phi^{i,j} + \dbar_{\cT} \phi^{i, j+1} = 0
\eeqn
for $0 \leq i , j \leq n$.
Here $\dbar_{\cT}$ denotes the differential internal to the dg Lie algebra~$\cT$. 
\item Cartan descent equations:
\beqn\label{eqn:cartandescent}
\partial \phi^{i,j} + \d_{CE} \phi^{i+1, j} = 0
\eeqn
for $0 \leq i , j \leq n$. 
Here $\d_{CE}$ denotes the Chevalley--Eilenberg differential associated to the Lie bracket of holomorphic vector fields.
\end{itemize}

In fact, the elements $\phi^{i,j}$ admit the following explicit expressions.
Define the degree $(-1)$ derivation $\Bar{\eta}_i$ of the dg Lie algebra $\cT(\C^n)$ by 
\beqn\label{eqn:holdescent}
\Bar{\eta}_i \left(\alpha(z,\zbar) \frac{\partial}{\partial z_j} \right) = \left(\iota_{\frac{\partial}{\partial \zbar_i}} \alpha \right) (z,\zbar) \frac{\partial}{\partial z_j} .
\eeqn
On the right-hand side, $i_X \alpha$ denotes the contraction of the differential form $\alpha$ by the vector field $X$. 
This derivation extends to a derivation of the algebra \eqref{eqn:deRhamdescent} that we denote by the same symbol. 

Next, define the derivation $\eta_i$ of the algebra $C^\bu(\cT(\C^n))$ by the formula
\beqn\label{eqn:cartandescent}
\eta_i (\psi) = \iota_{\frac{\partial}{\partial z_i}} \psi .
\eeqn
The right-hand side is the contraction of the cochain $\psi \in C^\bu(\cT(\C^n))$ by the vector field~$\frac{\partial}{\partial z_i}$. 
Explicitly, if $\psi$ is $k$-linear, then 
\beqn
(\iota_{\frac{\partial}{\partial z_i}} \psi)(\xi_1,\ldots, \xi_{k-1}) = \sum_{j=1}^{k-1} (-1)^{j+1}\psi\left(\xi_1,\ldots, \xi_j, \frac{\partial}{\partial z_i}, \xi_{j+1}, \ldots, \xi_{k-1} \right) .
\eeqn
This derivation also extends to a derivation of the algebra \eqref{eqn:deRhamdescent} that we denote by the same symbol. 

Given this notation, we return to the starting data which is a cochain $\phi \in C^\bu(\lie{vect}(n))$.
Recall, we set $\phi^0 = j^* \phi$ which is a zero form in the de Rham complex of jets. 
A representative for $\Phi$ as in the theorem is
\beqn\label{eqn:Phi}
\Phi = \exp\left(\sum_{i=1}^n \left(\d \zbar_i \Bar{\eta}_i + \d z_i \eta_i\right)\right) \phi^0 . 
\eeqn
Note that the derivations $\eta_i$ and $\Bar{\eta}_j$ commute for all $i,j$, so the right-hand side of the equation is unambiguously defined.

\begin{proof}[Proof of proposition \ref{prop:gfdescent}]
The differential on the de Rham complex $\Omega^\bu(\C^n , C^\bu(\op{jet}(\cT)))$ has the form $\d_{dR} + \d_{\cT}$ where $\d_{dR}$ is the de Rham differential encoded by the flat connection on $C^\bu(\op{jet}(\cT))$ and $\d_{\cT}$ is the differential internal to the complex $C^\bu(\cT)$. 
Note that $\d_{\cT}$ splits as $\d_{\cT} = \dbar_{\cT} + \d_{CE}$ where $\dbar_{\cT}$ is the $\dbar$-operator arising in the definition of $\cT$ (we use this notation to not confuse it with the de Rham differential), and $\d_{CE}$ is the differential arising from the Lie bracket on $\cT$. 

It suffices to show that the elements $\phi^{i,j}$ satisfy the pair of descent equations (\ref{eqn:holdescent}) and (\ref{eqn:cartandescent}). 
Since the operators $\eta_i$ and $\Bar{\eta}_i$ commute, it suffices to prove (\ref{eqn:holdescent}) for $i=0$ and (\ref{eqn:cartandescent}) for $j=0$. 

Note that $\phi^{0,j}$ is the $(0,j)$th component in the expansion of $\exp\left(\sum_\ell \d \zbar_\ell \Bar{\eta}_\ell \right) \phi^0$. 
For $i=0$, descent equation (\ref{eqn:holdescent}) follows from
\beqn
\dbar_{\cT} \d \zbar_\ell \Bar{\eta}_\ell \phi^{0, j} = - \d \zbar_\ell [\dbar_{\cT}, \Bar{\eta}_\ell] \phi^{0,j} = -\d \zbar_\ell \frac{\partial}{\partial \zbar_\ell} \phi^{0,j} .
\eeqn

For the second descent equation, note that $\phi^{i,0}$ is the $(i,0)$th component in the expansion $\exp\left(\sum_\ell \d z_\ell  \right) \phi^0$.
For $j=0$, descent equation (\ref{eqn:cartandescent}) follows from
\beqn
\d_{CE} \d z_\ell \eta_\ell \phi^{i,0} = - \d z_{\ell} [\d_{CE}, \eta_\ell] \phi^{i,0} = - \d z_\ell \frac{\partial}{\partial z_\ell} \phi^{i,0} .
\eeqn
\end{proof}

Combining this result with Theorem \ref{thm:hol} we obtain the following.

\begin{cor}
The composite map 
\beqn
\delta \colon C_{red}^\bu(\lie{vect}(n)) \to \Omega^\bu\left(\C^n , C_{red}^\bu(\op{jet} \cT) \right) \xto{\simeq} \cloc^\bu(\cT(\C^n)) [-2n].
\eeqn
which sends $\phi \mapsto \delta(\phi) = \int \phi^{n,n}$ is a quasi-isomorphism.
In particular, if $\phi \in C^\bu(\lie{vect}(n))$ is a Gelfand--Fuks cocycle of degree $k$, then $\delta (\phi) \in \cloc^\bu(\cT)$ is a local cocycle of degree $k-2n$ and up to equivalence all such local cocycles are obtained in this way.
\end{cor}

\subsection{Examples}
We spell out some low-dimensional examples.

\subsubsection{}
The reduced cohomology of one-dimensional formal vector fields is one-dimensional concentrated in degree $+3$.
Thus, by theorem \ref{thm:hol} we see
\beqn
H^1_{loc}(\cT(\C)) = H^3 (\lie{w}_1)
\eeqn
is one-dimensional.

A representative for this class can be deduced from the explicit Gelfand--Fuks cocycle in equation \eqref{eqn:GF1d}, that we will denote by $\phi$.
The section $\phi^0 = j^* \phi$ of $C^\bu(\op{jet}(\cT))$ is
\beqn
\phi^0 \bigg(\alpha(z,\zbar) \frac{\partial}{\partial z} , \beta(z,\zbar) \frac{\partial}{\partial z} , \gamma(z ,\zbar) \frac{\partial}{\partial z} \bigg) = {\rm det} \begin{pmatrix} \alpha^0 & \beta^0 & \gamma^0 \\ \partial_z\alpha^0 & \partial_z\beta^0 & \partial_z \gamma^0  \\ \partial_z^2 \alpha^0 & \partial_z^2\beta^0 & \partial_z^2 \gamma^0\end{pmatrix} (z,\zbar) .
\eeqn
Here $\alpha^0$ denotes the zero form component of the differential form $\alpha$. 
We first solve for the descent element $\phi^{0,1}$ which satisfies the holomorphic descent equation
\beqn
\dbar \phi^0 = \dbar_\cT \phi^{0,1} .
\eeqn
This element has the form $\phi^{0,1} = \d \zbar \psi^{0,1}$ where $\psi^{0,1}$ is the section of $C^\bu(\op{jet} (\cT))$ defined by
\beqn
\psi^{0,1} \bigg(\alpha \frac{\partial}{\partial z} , \beta \frac{\partial}{\partial z} , \gamma \frac{\partial}{\partial z} \bigg) = \phi^0\left(\alpha^{0,1} \frac{\partial}{\partial z} , \beta^0 \frac{\partial}{\partial z} , \gamma^0 \frac{\partial}{\partial z} \right) + \cdots
\eeqn
where $\cdots$ denotes the two terms obtained by swapping the role of $\alpha$ with $\beta,\gamma$ respectively. 
Next, we solve for $\phi^{1,1}$ which satisfies the Cartan descent equation
\beqn
\partial \phi^{0,1} = \d_{\rm CE} \phi^{1,1} .
\eeqn
Explicitly $\phi^{1,1} = \d z \d \zbar \psi^{1,1}$ with $\psi^{1,1}$ the section of $C^\bu(\op{jet} (\cT))$ defined by
\beqn
\psi^{1,1} \bigg(\alpha \frac{\partial}{\partial z} , \beta \frac{\partial}{\partial z} \bigg) = \partial_z \alpha^{0,1} \partial_z^2 \beta^0  - \partial_z^2 \alpha^{0,1} \partial_z \beta^0 + \left(\alpha \leftrightarrow \beta \right) .
\eeqn
The local cocycle $\delta(\phi) = \int \phi^{1,1} \in \cloc^\bu(\cT(\C))$ can be put in the following more uniform form
\beqn
\delta(\phi)(\mu) = \int J(\mu) \wedge \partial J(\mu) ,
\eeqn
where $J \mu = J(\alpha(z,\zbar) \del_z) = \del_z \alpha(z,\zbar)$ is the Jacobian of holomorphic vector fields extended to the Dolbeault resolution in the natural way.

More generally, we will use the following construction. 
If
\beqn
\alpha_i (z,\zbar) \frac{\del}{\del z_i} \in \Omega^{0,\bu}(\C^n , \T^{1,0})
\eeqn
is an element of the Dolbeault complex of the holomorphic tangent bundle, then we define
\beqn
J \mu \define \left( \frac{\del \alpha_i}{\del z_j} (z,\zbar) \right) \in \lie{gl}(n) \otimes \Omega^{0,\bu}(\C^n) .
\eeqn
%

\subsubsection{}

We have seen that $H^{1}_{loc}(\cT(\C^2))$ is two-dimensional corresponding to the two linearly independent classes in $H^5(\lie{vect}(2))$.
In turn, from the formal de Rham spectral sequence, these two classes arise from the classes $a_1 \tau_2, a_1 \tau_1 \tau_2 \in C^3(\lie{vect}(2), \Omega^2)$.
Using the explicit expressions for these classes given in section \ref{sec:gf} we arrive at the following.

\begin{prop}\label{prop:2dloc}
Explicit representatives for the two linearly independent classes in $H^{1}_{loc}(\cT(\C^2))$ are
\beqn
\int_{\C^2} \op{Tr} ( J \mu ) \op{Tr}(\del J \mu \del J \mu) ,
\eeqn
and
\beqn
\int_{\C^2} \op{Tr}(J \mu) \op{Tr}(\del J \mu) \op{Tr}(\del J \mu) .
\eeqn
\end{prop}

\subsubsection{}

In complex dimension one and two there are obvious parallels between the representatives of the local cocycle and the representative for the cocycle of formal vector fields.
Indeed, in complex dimension one we know that the generating class in $H^3(\lie{vect}(1))$ arises, via formal descent, from the class $a_1 c_1 \in H^2(\lie{vect}(1) ; \Omega^1)$.
Similarly, the local classes in proposition \ref{prop:2dloc} bare the same form as the representatives $a_1 \tau_2, a_1 \tau_1^2 \in H^3(\lie{vect}(n) ; \Omega^2)$ which, in turn, generate the cohomology $H^5(\lie{vect}(2))$ by applying formal descent.
The key feature that these examples have in common is that the generators of cohomology of $\lie{vect}(n)$ arise by applying formal descent to classes valued in the top formal de Rham forms $\Omega^n(\Hat{D}^n)$. 

In complex dimension three, however, there are classes in $H^7(\lie{vect}(3))$ which do not arise from applying formal descent to a Gelfand--Fuks class with values in $\Omega^3(\Hat{D}^3)$.
Indeed, $H^7(\lie{vect}(3))$ is four-dimensional. 
The classes which generate the cohomology, via formal descent, are $a_1 \tau_1^3, a_1 \tau_1 \tau_2 ,a_1 \tau_3$ and $a_2 \tau_2$.
The first three elements live in $H^4(\lie{vect}(3) ; \Omega^3)$ while the latter element is the generator of $H^5(\lie{vect}(3) ; \Omega^2)$ which takes the explicit form
\beqn
a_2 \tau_2 \colon (X_0,\ldots,X_4) \mapsto \op{Tr} (JX_0 JX_1 J X_2) \op{Tr} (\d J X_3 \d X_4) + \cdots .
\eeqn
Up to a factor, the corresponding local cocycle is
\beqn
\int_{\C^3} \op{Tr}(J \mu \del J \mu) \op{Tr}(\del J \mu \del J \mu) .
\eeqn

\printbibliography

\end{document}